\documentclass[12pt]{amsart}

\usepackage{amsmath,amssymb,amsfonts,amscd,latexsym}
\usepackage[all]{xy}
\usepackage{graphics,epsfig,colortbl}
\input xy
\xyoption{all}

\usepackage{vmargin} 
\setmarginsrb{20mm}{10mm}{20mm}{16mm}{20mm}{7mm}{20mm}{8mm}

\usepackage[pdftex,bookmarks,colorlinks,breaklinks]{hyperref} 

\newtheorem{theorem}{Theorem}[section]

\newtheorem{proposition}[theorem]{Proposition}
\newtheorem{corollary}[theorem]{Corollary}
\newtheorem{lemma}[theorem]{Lemma}
\newtheorem{claim}{Claim}

\newtheorem{question}[theorem]{Question}

\theoremstyle{definition}
\newtheorem{definition}[theorem]{Definition}

\theoremstyle{remark}

\newtheorem{remark}[theorem]{Remark}

\newcommand{\T}{\mathbb{T}}

\newcommand{\N}{\mathbb{N}}

\newcommand{\LIM}{\mathbf{LIM}}

\newcommand{\hull}[1]{\langle #1 \rangle}

\def\c {{\mathfrak{c}}}

\title[$d$-independent topological groups]{Remarks on $d$-independent topological groups}

\author[Z. Huang]{Zhouxiang Huang}
\address{Institute of Mathematics, Nanjing Normal University, Nanjing 210024, China}
\email{2715383574@qq.com}

\author[D. Peng]{Dekui Peng\textsuperscript{*}}
\address{Institute of Mathematics, Nanjing Normal University, Nanjing 210024, China}
\email{pengdk10@lzu.edu.cn}
\thanks{*Corresponding author.}

\author[G. Zhang]{Gao Zhang}
\address{ School of Mathematical Sciences, Jiangsu Second Normal University, 211222, China}
\email{gaozhang0810@hotmail.com}

\begin{document}

\begin{abstract}
A non-trivial topological group is called \emph{$d$-independent} if for every subgroup of cardinality less than the continuum there exists a countable dense subgroup intersecting it trivially.
This notion was introduced by Márquez and Tkachenko and has been intensively studied in the metrizable setting.
In particular, they proved that a second-countable locally compact abelian group is $d$-independent if and only if it is algebraically an $M$-group, and asked whether the same conclusion holds for all separable locally compact groups.

In this paper we give an affirmative answer to this question.
We show that every separable locally compact abelian $M$-group is $d$-independent, thereby removing the metrizability assumption from the result of Márquez and Tkachenko.

In addition, we investigate several further aspects of $d$-independence.
We study its behaviour under taking powers of topological groups and extend the notion of $d$-independence to the non-abelian setting.
Moreover, we prove that every separable connected compact group is $d$-independent, thereby answering another question posed by Márquez and Tkachenko.

\end{abstract}

\maketitle

\section{Introduction}
\label{sec0}

Separability is without doubt one of the most fundamental and widely studied notions in general topology. 
Nevertheless, separability behaves poorly in general: even a closed subspace of a separable compact space need not be separable.
In contrast, the situation improves substantially in the setting of topological groups.
It is well known that every closed subgroup of a separable locally compact group is separable \cite{CI}.
This result was first established by Itzkowitz in the compact case \cite{It}, and later generalized by Leiderman, Morris and Tkachenko, who proved that every closed subgroup of a separable feathered group is separable \cite{LMT}.

Motivated by these results, Leiderman and Tkachenko introduced and studied the class $\mathbf{CSS}$ of topological groups in which all closed subgroups are separable \cite{LT}.
Closely related problems have also been considered from a different perspective; for instance, the second listed author investigated topological groups whose dense subgroups are all separable \cite{Pe}.

In a subsequent development, Xiao, Sánchez and Tkachenko \cite{XST} observed an intriguing phenomenon.
They showed that for certain standard topological groups $G$, and for every subgroup $S\leq G$ of cardinality less than $\mathfrak{c}$, there exists a countable dense subgroup $D$ of $G$ such that $D\cap S=\{1\}$.
This property was further studied systematically by Márquez and Tkachenko in \cite{MT}, where such groups were termed \emph{$d$-independent topological groups}.

\begin{definition}A non-trivial topological group $G$ is called \emph{$d$-independent}, if for any subgroup $S\leq G$ with $|S|<\c$, there exists a countable dense subgroup $D$ of $G$ such that $S\cap D=\{1\}$.\end{definition}

Among other results, Márquez and Tkachenko proved that a second-countable topological group is $d$-independent if and only if it is \emph{maximally fragmentable}, a notion introduced earlier by Comfort and Dikranjan in their study of dense subgroups of topological groups.
They also obtained a complete characterization of $d$-independence within the class of second-countable locally compact abelian groups.
More precisely, they showed that:

\begin{quote}
A second-countable locally compact abelian group is $d$-independent if and only if it is algebraically an $M$-group.
\end{quote}

Here, an abelian group $G$ is called an \emph{$M$-group} if for every integer $n$, the cardinality of $nG$ is either $1$ or at least $\mathfrak{c}$.
This class of groups was introduced earlier by Dikranjan and Shakhmatov \cite{DS} in connection with a long-standing problem of Markov.
It is worth noting that being an $M$-group is a necessary condition for an abelian topological group to be $d$-independent.

In \cite{MT}, the sufficiency of this condition was established under the assumption of second countability, and the proof relies heavily on techniques from the theory of spaces with countable tightness.
However, every compact group of countable tightness is metrizable, which explains why their method applies only in the metrizable setting.
This led Márquez and Tkachenko to pose the following question:

\begin{question}
Is a separable locally compact group $d$-independent whenever it is an $M$-group?
\end{question}

The main result of the present paper, proved in Section~\ref{S2}, gives an affirmative answer to this question.
Our approach is largely independent of the methods developed in \cite{MT} and, moreover, yields several additional generalizations of their results.

In the final section of the paper, we discuss two further aspects of $d$-independence.
First, we consider the behaviour of $d$-independence under taking powers.
Secondly, we answers another question of Márquez and Tkachenko by showing that every separable connected compact group is $d$-independent.

\subsection{Notation, terminology and preliminary facts}

Throughout this paper, all topological groups are assumed to be Hausdorff.
We use the symbols $\mathbb{Z}$, $\mathbb{Q}$, $\mathbb{R}$ and $\mathbb{T}$
to denote the additive group of integers, the additive group of rational numbers,
the additive group of real numbers, and the unit circle group, respectively.
The sets of natural numbers and positive integers are denoted by $\N$ and $\N_{+}$.

For each $n\in \N_{+}$, we denote by $\mathbb{Z}(n)$ the cyclic group of order $n$,
and by $\mathbb{Z}(p^\infty)$ the Prüfer $p$-group, where $p$ is a prime.
The group of $p$-adic integers, which is the Pontryagin dual of
$\mathbb{Z}(p^\infty)$, is denoted by $\mathbb{Z}_p$.

The cardinality of a set $A$ is denoted by $|A|$.
The symbols $\omega$ and $\c$ stand for the cardinals $|\mathbb{N}|$ and
$|\mathbb{R}|$, respectively.

Abelian groups are written additively.
The identity element of an abelian group $G$ is denoted by $0_G$, or simply by
$0$ when no confusion may arise.
Non-abelian groups are considered only in Section~\ref{S3}; in this case, the
identity element of a group $G$ is denoted by $1_G$ or $1$.
For a subset $X\subseteq G$, the subgroup of $G$ generated by $X$ is denoted by
$\hull{X}$.

An element $g$ of a group $G$ is said to be of \emph{finite order}, or equivalently,
a \emph{torsion element}, if $g^n=1_G$ for some $n\in\N_{+}$.
A group $G$ is called a \emph{torsion group} if all its elements are torsion.
If $G$ has no torsion elements other than the identity, then $G$ is called
\emph{torsion-free}.
A torsion group is said to be \emph{bounded} if there exists a positive integer
$m$ such that $g^m=1_G$ for all $g\in G$.


A family $\{H_i: i\in I\}$ of subgroups of an abelian group $G$, where the index
set $I$ is well ordered, is called \emph{independent} if
\[
H_i\cap \sum_{j\in I\setminus\{i\}} H_j=\{0\}
\quad\text{for every } i\in I.
\]
It is easy to verify that this condition is equivalent to
\[
H_i\cap \sum_{j<i} H_j=\{0\}
\quad\text{for every } i\in I,
\]
and also equivalent to the fact that the sum $\sum_{i\in I} H_i$ is direct.
A family $\{H_i:i\in I\}$ of subgroups is called \emph{pairwise independent} if $\{H_i, H_j\}$ is independent, or equivalently, if $H_i\cap H_j=\{0\}$,  for any distinct $i,j\in I$.

For a subset $X\subseteq G$, we say that $X$ consists of \emph{independent elements}
if the family $\{\hull{x}: x\in X\}$ is independent.

The following elementary but useful lemma concerns independent families of
subgroups.

\begin{lemma}\label{indfam}
Let $G$ be an abelian group and let $H$ and $K$ be subgroups of $G$ such that
$H\cap K=\{0\}$.
Let $\mathcal{A}$ be an independent family of subgroups of $H$.
For each $A\in \mathcal{A}$, fix a homomorphism $f_A\colon A\to K$, and define
\[
B_A=\{a+f_A(a): a\in A\}.
\]
Then the family $\mathcal{B}=\{B_A: A\in \mathcal{A}\}$ is independent.
\end{lemma}

\begin{proof}
Let $A_1,\dots,A_n\in \mathcal{A}$ and choose $a_i\in A_i$ for each
$i=1,\dots,n$.
Assume that
\[
\sum_{i=1}^n \bigl(a_i+f_{A_i}(a_i)\bigr)=0.
\]
Since $a_i\in H$ and $f_{A_i}(a_i)\in K$, and $H\cap K=\{0\}$, we obtain
\[
\sum_{i=1}^n a_i=0
\qquad\text{and}\qquad
\sum_{i=1}^n f_{A_i}(a_i)=0.
\]
As the family $\mathcal{A}$ is independent, the first equality implies that
$a_i=0$ for all $i=1,\dots,n$.
This completes the proof.
\end{proof}

Undefined notions and symbols related to topological groups can be found in
\cite{ADG,AT}.

\section{$d$-Independent locally compact abelian groups}\label{S2}

We first present a simple lemma characterizing $d$-independent abelian groups in terms of independent dense subgroups.

\begin{lemma}\label{Le:eq}
Let $G$ be a separable abelian topological group. Then the following are equivalent:
\begin{itemize}
\item[(1)] $G$ is $d$-independent;
\item[(2)] $G$ admits a family $\{D_\alpha:\alpha<\mathfrak c\}$ of independent countable dense subgroups;
\item[(3)] $G$ admits a family $\{D_\alpha:\alpha<\mathfrak c\}$ of pairwise independent countable dense subgroups.
\end{itemize}
\end{lemma}

\begin{proof}
(1)$\Rightarrow$(2).
Assume that $G$ is $d$-independent and fix a countable dense subgroup $D_0$ of $G$.
By transfinite induction, for each ordinal $\alpha<\mathfrak c$ we choose a countable dense subgroup $D_\alpha\leq G$ such that
\[
D_\alpha\cap \sum_{\beta<\alpha} D_\beta=\{0\}.
\]
This is possible since $G$ is $d$-independent and
\(
\bigl|\sum_{\beta<\alpha} D_\beta\bigr|<\mathfrak c.
\)
It follows that $\{D_\alpha:\alpha<\mathfrak c\}$ is a family of independent countable dense subgroups of $G$.

\medskip
(2)$\Rightarrow$(3).
This implication is immediate.

\medskip
(3)$\Rightarrow$(1).
Let $\{D_\alpha:\alpha<\mathfrak c\}$ be a family of pairwise independent countable dense subgroups of $G$, and let $S\leq G$ with $|S|<\mathfrak c$.
Since the subgroups $D_\alpha$ are pairwise independent, each nonzero element $s\in S$ belongs to at most one $D_\alpha$.
Hence there exists some $\kappa<\mathfrak c$ such that
\[
D_\kappa\cap S=\{0\},
\]
showing that $G$ is $d$-independent.
\end{proof}

\begin{theorem}\label{Th:eq}
A separable abelian group $G$ with a torsion-free $d$-independent subgroup
is itself $d$-independent.
\end{theorem}

\begin{proof}
Let $H$ be a torsion-free $d$-independent subgroup of $G$.
Since $G$ is separable, there exists a dense countable subgroup $K$ of $G$.
Set $S=K\cap H$, which is countable.

As $H$ is $d$-independent, we can choose a dense countable subgroup
$D_0$ of $H$ such that $S\cap D_0=\{0_G\}$.
By transfinite induction, for each ordinal $\alpha<\mathfrak{c}$ we may select
a dense countable subgroup $D_\alpha$ of $H$ such that
\[
D_\alpha\cap\Bigl(S+\sum_{\beta<\alpha}D_\beta\Bigr)=\{0_G\}.
\]
Hence the family
\[
\mathcal{A}:=\{D_\alpha:\alpha<\mathfrak{c}\}
\]
is independent and $A\cap K=\{0_G\}$, where $A=\bigoplus_{\alpha<\c} D_\alpha$.

Let $\mathbf{LIM}$ denote the set of all limit ordinals below $\mathfrak{c}$.
For each $\alpha<\mathfrak{c}$, choose a nonzero element
$b_\alpha\in D_\alpha$.
For every $\gamma\in\mathbf{LIM}$, define
\[
P_\gamma=\sum_{n\in\omega\setminus\{0\}}\langle b_{\gamma+n}\rangle.
\]
%
By the construction of the family $\{D_\alpha:\alpha<\mathfrak{c}\}$, this sum
is direct; hence $P_\gamma$ is a free abelian group of rank $\omega$.

Fix, for each $\gamma\in\mathbf{LIM}$, a surjective homomorphism
$f_\gamma\colon D_\gamma\oplus P_\gamma\to K$ such that
$D_\gamma\leq\ker f_\gamma$, and set
\[
Q_\gamma=\{x+f_\gamma(x):x\in D_\gamma\oplus P_\gamma\}.
\]

We claim that each $Q_\gamma$ is dense in $G$.
Indeed, $Q_\gamma\cap H$ is dense in $H$ since it contains $D_\gamma$.
Therefore $\overline{H}\leq\overline{Q_\gamma}$, and in particular
$D_\gamma\oplus P_\gamma\leq\overline{Q_\gamma}$.
Thus, for every $x\in D_\gamma\oplus P_\gamma$,
\[
f_\gamma(x)=(x+f_\gamma(x))-x\in Q_\gamma+\overline{Q_\gamma}
\subseteq\overline{Q_\gamma}.
\]
As $f_\gamma$ is surjective, it follows that $K\leq\overline{Q_\gamma}$.
Since $K$ is dense in $G$, we conclude that $Q_\gamma$ is dense in $G$.

The independence of the family
$\{Q_\gamma: \gamma\in \LIM\}$
follows directly from Lemma~\ref{indfam}, as $Q_\gamma\leq A$.
Therefore, by Lemma~\ref{Le:eq}, the group $G$ is $d$-independent.
\end{proof}

\begin{lemma}\label{infip}
Let $G$ be an unbounded totally disconnected compact abelian group.
If $G$ contains no subgroup topologically isomorphic to $\mathbb{Z}_p$ for any prime $p$,
then $G$ has a closed subgroup of the form
\[
\prod_{i\in \omega} \mathbb{Z}(p_i),
\]
where $\{p_i : i\in \omega\}$ is a family of pairwise distinct primes.
\end{lemma}

\begin{proof}
Let $X$ be the Pontryagin dual group of $G$.
Then $X$ is a torsion abelian group and is not bounded.
There are two possible cases:

\begin{enumerate}
\item[(1)] There exists a prime $p$ such that the $p$-primary component $X_p$ of $X$ is unbounded;
\item[(2)] There exist infinitely many primes $p$ such that $X_p$ is non-trivial.
\end{enumerate}

We first show that case~(1) cannot occur.

Assume towards a contradiction that $X_p$ is unbounded for some prime $p$.
By a classical theorem of Fuchs \cite[Theorem~32.3]{Fuc},
the group $X_p$ admits a \emph{basic subgroup} $B_p$; that is,
$B_p$ is a pure subgroup of $X_p$ which is a direct sum of cyclic groups
and such that $X_p/B_p$ is divisible.

If $X_p/B_p$ is non-trivial, then it contains a subgroup isomorphic to the Prüfer group
$\mathbb{Z}(p^\infty)$.
Consequently, $X$, and hence $X_p$, admits a quotient group isomorphic to $\mathbb{Z}(p^\infty)$.
By Pontryagin duality, this implies that $G$ contains a closed subgroup
topologically isomorphic to $\mathbb{Z}_p$, contradicting the hypothesis.

Therefore, $X_p=B_p$ is itself a direct sum of cyclic groups.
Write
\[
X_p=\bigoplus_{s\in S} Y_s,
\]
where each $Y_s$ is cyclic.
Since each $Y_s$ embeds into $\mathbb{Z}(p^\infty)$ and $X_p$ is their direct sum,
there exists a homomorphism from $X_p$ into $\mathbb{Z}(p^\infty)$ extending all these embeddings.
As $X_p$ is unbounded, this homomorphism must be surjective.
Hence $\mathbb{Z}(p^\infty)$ is again a quotient group of $X_p$, and therefore of $X$.

Dualizing once more, we conclude that $G$ contains a closed subgroup
topologically isomorphic to $\mathbb{Z}_p$, a contradiction.

For case~(2), choose pairwise distinct primes $\{p_i:i\in\omega\}$ such that
$X_{p_i}\neq\{0\}$ for each $i$.
Since case~(1) does not occur, every $X_{p_i}$ is bounded.
Hence each $X_{p_i}$ admits a quotient group isomorphic to $\mathbb{Z}(p_i)$.
It follows that $X$ admits a quotient group isomorphic to
\[
\bigoplus_{i\in\omega}\mathbb{Z}(p_i).
\]
By Pontryagin duality, $G$ therefore contains a closed subgroup
topologically isomorphic to
$\prod_{i\in\omega}\mathbb{Z}(p_i).$
\end{proof}

\begin{corollary}\label{unbounded} 
A separable abelian group $G$ with an unbounded compact subgroup is itself
$d$-independent.
\end{corollary}

\begin{proof}
By Theorem~\ref{Th:eq}, it suffices to show that every unbounded separable
compact abelian group $K$ contains a torsion-free $d$-independent subgroup
(note that every compact subgroup of $G$ is necessarily separable).

First assume that $K$ is not totally disconnected.
By \cite[Lemma~2.37]{MT}, the identity component $K_0$ of $K$ admits a family
of cardinality $\mathfrak{c}$ consisting of independent dense cyclic
subgroups, which are necessarily torsion-free.
The direct sum of these subgroups is therefore a free abelian group, and it
is $d$-independent by Lemma~\ref{Le:eq}.

Next, assume that $K$ contains a compact subgroup topologically isomorphic to
the $p$-adic group $\mathbb{Z}_p$ for some prime $p$.
Since $\mathbb{Z}_p$ is torsion-free and $d$-independent, the conclusion
follows immediately.

It remains to consider the case where $K$ is totally disconnected and
contains no subgroup topologically isomorphic to $\mathbb{Z}_p$ for any
prime $p$.
In this situation, according to Lemma \ref{infip}, $K$ contains a subgroup of the form
\[
N\cong\prod_{i=1}^{\infty}\mathbb{Z}(p_i),
\]
where $p_i$ is an odd prime for each $i$.
So, it suffices to assume that $K=N$.
We show that $K$ contains a torsion-free $d$-independent subgroup.

Let $X$ be the Pontryagin dual group of $K$.
Then $X$ is isomorphic to $\bigoplus_{i=1}^{\infty}\mathbb{Z}(p_i)$.
For each $i$, let $X_i$ denote the subgroup of $X$ isomorphic to
$\mathbb{Z}(p_i)$, and fix a nonzero element $x_i\in X_i$.

Let $[\omega]^{\omega}$ denote the family of all infinite subsets of $\omega$.
It is well known that $[\omega]^{\omega}$ admits a chain $\mathcal{C}$ of cardinality
$\mathfrak{c}$ such that the difference of any two distinct elements
of $\mathcal{C}$ is infinite.
Indeed, identifying $\omega$ with $\mathbb{Q}$, for each real number $x$ one may
consider the set
\[
R_x:=\{r\in\mathbb{Q}: r<x\}.
\]
The family $\{R_x:x\in\mathbb{R}\}$ forms such a chain.

Fix such a chain $\mathcal{C}\subseteq[\omega]^\omega$.
For each $E\in\mathcal{C}$, define a homomorphism
$\chi_E\colon X\to\mathbb{T}$
by\footnote{Here we view $\T$ as the continuous image of the interval $(-\frac{1}{2}, \frac{1}{2}]$ under the natural map $\psi$, and for each $x \in \T$, we identify it with the unique real number $y \in (-\frac{1}{2}, \frac{1}{2}]$ such that $\psi(y) = x$.}
\[
\chi_E(x_n)=
\begin{cases}
\dfrac{1}{p_n}, & n\in E,\\[6pt]
-\dfrac{1}{p_n}, & n\notin E.
\end{cases}
\]
Each $\chi_E$ induces a group embedding of $X$ into $\mathbb{T}$ with dense
image.
By standard arguments in Pontryagin duality, when $\chi_E$ is regarded in the
canonical way as an element of $K$, it generates a dense torsion-free subgroup
of $K$.

To prove that $K$ is $d$-independent, it suffices by Lemma~\ref{Le:eq} to show that the family
\[
\{\chi_E:\; E\in\mathcal{C}\}
\]
is pairwise independent. In fact, we shall establish the stronger statement that this family is independent, which is of independent interest.

Indeed, let
\[
E_1\subsetneq E_2\subsetneq \cdots \subsetneq E_m
\]
be elements of $\mathcal{C}$ and suppose that
\[
\sum_{i=1}^m k_i\chi_{E_i}=0
\qquad (k_1,\dots,k_m\in\mathbb{Z}).
\]

Put $E_0=\emptyset$.
Since each difference $E_j\setminus E_{j-1}$ is infinite, for every
$j\in\{1,\dots,m\}$ we may choose
\[
s_j\in E_j\setminus E_{j-1}
\quad\text{such that}\quad
p_{s_j}>\sum_{i=1}^m |k_i|.
\]
Then
\[
0=\Bigl(\sum_{i=1}^m k_i\chi_{E_i}\Bigr)(x_{s_j})
 =\sum_{i=1}^m k_i\,\chi_{E_i}(x_{s_j})
 =\frac{-\sum_{i=1}^{j-1} k_i+\sum_{i=j}^m k_i}{p_{s_j}}.
\]
By the choice of $s_j$, the absolute value of the numerator is strictly smaller
than $p_{s_j}$.
Hence this element is equal to $0$ in $\mathbb{T}$ if and only if
\[
-\sum_{i=1}^{j-1} k_i+\sum_{i=j}^m k_i=0.
\]

Letting $j$ range over $\{1,\dots,m\}$, we obtain a system of $m$ linear
equations whose unique solution is
\[
k_1=k_2=\cdots=k_m=0.
\]
Therefore, the family $\{\chi_E:E\in\mathcal{C}\}$ is independent.

Finally, since $|\mathcal{C}|=\mathfrak{c}$, Lemma~\ref{Le:eq} implies that
\[
S:=\bigoplus_{E\in\mathcal{C}}\langle\chi_E\rangle\leq K
\]
is $d$-independent.
\end{proof}

We next present a bounded counterpart of Theorem~\ref{Th:eq}.
Although the proof is very similar, we include it here for the sake of
completeness.
Recall that for a bounded abelian group $G$, the \emph{exponent} of $G$ is the
least positive integer $n$ such that $nG=\{0\}$.

\begin{theorem}\label{bounded}
Let $G$ be a bounded separable topological abelian group of exponent $n$.
If $G$ admits a $d$-independent subgroup $H$ of exponent $n$, then $G$ is
$d$-independent.
\end{theorem}

\begin{proof}
Let $K$ be a countable dense subgroup of $G$.
Arguing as in the proof of Theorem~\ref{Th:eq}, one can construct a family
$\{D_\alpha:\alpha<\mathfrak{c}\}$ of independent countable dense
subgroups of $H$ such that
\[
\Bigl(\bigcup_{\alpha<\mathfrak{c}}D_\alpha\Bigr)\cap K=\{0\}.
\]

For each $\alpha<\mathfrak{c}$, choose an element $b_\alpha\in D_\alpha$ of
order $n$.
Since $K$ is countable and has exponent $n$, for every limit ordinal
$\gamma<\mathfrak{c}$ the group
\[
\bigoplus_{m\in\omega\setminus\{0\}}\langle b_{\gamma+m}\rangle
\]
admits a surjective homomorphism onto $K$.
Consequently, there exists a surjective homomorphism
\[
f_\gamma\colon
D_\gamma\oplus\Bigl(\bigoplus_{m\in\omega\setminus\{0\}}
\langle b_{\gamma+m}\rangle\Bigr)\longrightarrow K
\]
whose kernel contains $D_\gamma$.

As in the proof of Theorem~\ref{Th:eq}, consider the graphs of these maps and
define
\[
Q_\gamma:=\{x+f_\gamma(x):
x\in D_\gamma\oplus(\bigoplus_{m\in\omega\setminus\{0\}}
\langle b_{\gamma+m}\rangle)\},
\qquad \gamma\in\mathbf{LIM}.
\]
Then the family $\{Q_\gamma:\gamma\in\mathbf{LIM}\}$ consists of independent
countable dense subgroups of $G$.
By Lemma~\ref{Le:eq}, this implies that $G$ is $d$-independent.
\end{proof}

We are now ready to give the final proof of the main result.

\begin{theorem}
Let $G$ be a separable locally compact abelian group.
Then $G$ is $d$-independent if and only if $G$ is an $M$-group.
\end{theorem}

\begin{proof}
Only the ``if'' direction requires verification since $d$-independent abelian groups are $M$-groups \cite[Proposition 2.1]{MT}.
By the structure theorem for locally compact abelian groups,
$G$ is topologically isomorphic to $\mathbb{R}^m\times H$ for some
$m\in\omega$, where $H$ admits an open compact subgroup $K$.

If $m>0$ or if $K$ is unbounded, then Theorem \ref{Th:eq} or Corollary~\ref{unbounded} implies that
$G$ is $d$-independent.
Thus it remains to consider the case $m=0$, that is, $G=H$, and $K$ is
bounded.

Let $k$ be a positive integer.
Then $kK\leq kG$, and
\[
kG/kK\cong k(G/K)\leq G/K.
\]
Since $G$ is separable and $K$ is open in $G$, the quotient $G/K$ is
countable.
Hence $kG/kK$ is also countable.
By the definition of an $M$-group, the group $kG$ has cardinality either $1$
or at least $\mathfrak{c}$; the same therefore holds for $kK$.
In particular, $K$ is itself an $M$-group.

Let $n$ be the exponent of $K$, and take $k=n$.
Then $kG$ is countable, hence trivial.
It follows that $G$ and $K$ have the same exponent.

Recall that every bounded compact abelian group admits a decomposition as a
direct product of finite abelian $p$-groups.
Fix a prime $p$ dividing $n$, and let $s$ be the largest integer such that
$p^s$ divides $n$.
In the direct product decomposition of $K$, there are infinitely many
factors isomorphic to $\mathbb{Z}(p^s)$.
Indeed, the group $\frac{n}{p}K$ is nontrivial and hence infinite, while in
the homomorphism
\[
K\to \tfrac{n}{p}K,\qquad x\mapsto \tfrac{n}{p}x,
\]
all other primary components vanish.
Consequently, $K$ contains a closed subgroup isomorphic to
$\mathbb{Z}(p^s)^\omega$.

Letting $p$ range over all prime divisors of $n$, the sum of the resulting
subgroups is isomorphic to $\mathbb{Z}(n)^\omega$, which is $d$-independent by \cite[Theorem 3.7]{MT}.
Applying Corollary~\ref{bounded}, we conclude that $G$ is $d$-independent.
\end{proof}

\section{Further remarks}\label{S3}

In this section, we discuss several additional problems and observations related to $d$-independence.

\subsection{Powers of $d$-independent groups}
\begin{proposition}\label{Prop:InfP}
Let $G$ be a non-trivial separable topological abelian group and let
$\kappa$ be an infinite cardinal with $\kappa\leq \mathfrak{c}$.
Then $G^\kappa$ is $d$-independent.
\end{proposition}

\begin{proof}
Let $H$ be a countable dense subgroup of $G$.
By definition of $d$-independence, if a topological group contains a dense
$d$-independent subgroup, then it is itself $d$-independent.
Since $H^\kappa$ is dense in $G^\kappa$, it suffices to show that
$H^\kappa$ is $d$-independent.

As $H$ is countable, it has a countable network.
Hence we may apply \cite[Theorem~3.7]{MT}, which asserts that $X^\kappa$
is $d$-independent for every non-trivial topological abelian group $X$
with a countable network.
\end{proof}

\begin{proposition}\label{n+1}
Let $G$ be a separable topological abelian group such that $G^n$ is
$d$-independent for some positive integer $n$.
Then $G^{n+1}$ is also $d$-independent.
\end{proposition}

\begin{proof}
By Lemma~\ref{Le:eq}, the group $G^n$ admits a family
$\{H_\alpha:\alpha<\mathfrak{c}\}$ of independent countable dense subgroups.
Let $p\colon G^n\to G$ be the projection onto the first coordinate.
Then $p(H_\alpha)$ is dense in $G$ for each $\alpha<\mathfrak{c}$.

Let $\mathbf{LIM}$ denote the set of all limit ordinals below $\mathfrak{c}$.
For each $\gamma\in\mathbf{LIM}$, define a homomorphism
$f_\gamma\colon H_\gamma\oplus H_{\gamma+1}\to G$ by setting
\[
f_\gamma|_{H_\gamma}=p|_{H_\gamma}
\quad\text{and}\quad
f_\gamma(H_{\gamma+1})=\{0\}.
\]

Consider the group $P=P_0\oplus G\cong P_0\times G$, where $P_0$ is a copy of $G^n$ endowed
with the product topology.
Clearly, $P_0\cap G=\{0\}$.
Hence, by Lemma~\ref{indfam}, the family
\[
D_\gamma:=\{x+f_\gamma(x):x\in H_\gamma\oplus H_{\gamma+1}\},
\qquad \gamma\in\mathbf{LIM},
\]
is independent.

Since each $D_\gamma$ is countable, by Lemma~\ref{Le:eq} it suffices to show
that $D_\gamma$ is dense in $P$ for every $\gamma\in\mathbf{LIM}$.
The argument is analogous to that in the proof of Theorem~\ref{Th:eq}.
Indeed, $D_\gamma$ contains $H_{\gamma+1}$, which is dense in $P_0$.
Therefore $\overline{D_\gamma}$ contains $P_0$.

Moreover, for every $x\in H_\gamma$, we have
$f_\gamma(x)\in\overline{D_\gamma}$.
Since
\[
f_\gamma(H_\gamma)=p(H_\gamma)
\]
is dense in $G$, it follows that $\overline{D_\gamma}$ also contains $G$.
Consequently,
\[
P=P_0\oplus G\subseteq \overline{D_\gamma},
\]
and hence $D_\gamma$ is dense in $P$.
This completes the proof.
\end{proof}

\begin{remark}
By the above proposition, one may define a function $\varphi$ on the class of
all separable topological abelian groups with values in
$\omega + 1=\omega\cup\{\omega\}$, by letting $\varphi(G)$ be the least cardinal
$\kappa$ such that $G^\kappa$ is $d$-independent.
In particular, $G$ is $d$-independent if and only if $\varphi(G)=1$.

However, it is not clear whether this function is of genuine interest.
More precisely, it is still unknown whether the $d$-independence of $G^n$
for some $n\geq 1$ implies the $d$-independence of $G$ itself.
If this were the case, then $G$ would be $d$-independent if and only if
$G^n$ is $d$-independent for every $n\geq 1$, and consequently the function
$\varphi$ would take only two possible values, namely $1$ and $\omega$.

This leads to the following question.
\end{remark}

\begin{question}
Let $G$ be a topological abelian group such that $G\times G$ is
$d$-independent. Is $G$ necessarily $d$-independent?
\end{question}

If the answer to this question is positive, then $\varphi$ also takes only two
values.
Indeed, suppose that $\varphi(G)=n<\omega$, and let $m$ be the least integer
such that $2m\geq n$.
Then $m<n$ whenever $n>1$.
Since $G^n$ is $d$-independent, so is $G^{2m}$, by Proposition \ref{n+1}.
A positive answer to the above question would then imply that $G^m$ is
$d$-independent.
By the definition of $\varphi$, this yields $\varphi(G)\leq m<n$, a
contradiction unless $n=1$.

\subsection{$d$-Independence of Non-abelian Groups}

Independent subgroups play a central role in the study of abelian
$d$-independent groups. It is therefore natural to extend this notion
to the non-abelian setting.

Let $G$ be a group and let $(H_\alpha)_{\alpha<\kappa}$ be a chain of
subgroups of $G$. We say that this chain consists of
\emph{independent subgroups} if
\[
H_\alpha \cap \hull{\bigcup_{\beta<\alpha} H_\beta}=\{1\}
\quad \text{for every } \alpha<\kappa,
\]
where $1$ denotes the identity element of $G$.

It is straightforward to verify that this definition coincides with
the usual notion of independence in the abelian case. The main
difference in the non-abelian setting is that the order of the chain
matters: unlike in the abelian case, independence is not invariant
under permutations of the subgroups.

For example, let
\[
G=(\mathbb{R}\times \mathbb{R})\rtimes \mathbb{Z}(2),
\]
where the non-trivial element of $\mathbb{Z}(2)$ acts on
$\mathbb{R}\times \mathbb{R}$ by exchanging the coordinates, that is,
by sending $(x,y)$ to $(y,x)$.
Let $R_1$ and $R_2$ denote the first and second coordinate subgroups of
$\mathbb{R}\times \mathbb{R}$, respectively, and let $R_3$ be the unique
subgroup of $G$ isomorphic to $\mathbb{Z}(2)$.
Then the chain $(R_1,R_2,R_3)$ is independent, whereas
$(R_1,R_3,R_2)$ is not.

Pairwise independence can be defined in a similar manner; note,
however, that pairwise independence does not depend on the order of
the subgroups.

We now give a non-abelian analogue of Lemma~\ref{Le:eq}, although it will not be used elsewhere in this paper.

\begin{lemma}\label{Le:eq2}
Let $G$ be a separable topological group. Then the following are equivalent:
\begin{itemize}
\item[(1)] $G$ is $d$-independent;
\item[(2)] $G$ admits a chain $(D_\alpha)_{\alpha<\c}$ of independent countable dense subgroups;
\item[(3)] $G$ admits a family $\{D_\alpha:\alpha<\mathfrak c\}$ of pairwise independent countable dense subgroups.
\end{itemize}
\end{lemma}

\begin{proof}
The proof is analogous to that of Lemma~\ref{Le:eq}, with the notion of
independence interpreted in the non-abelian sense.
\end{proof}

It is an easy consequence of Pontryagin duality that every separable connected compact
abelian group admits a family of cardinality $\mathfrak{c}$ of independent dense cyclic
subgroups; this fact has already been used in \cite{MT}.
We shall show that such a family can be chosen with extra properties.

\begin{lemma}\label{Le:con}
Let $G$ be a separable connected compact abelian group.
Then $G$ admits a family $\{H_\alpha : \alpha<\mathfrak{c}\}$ of independent dense cyclic
subgroups such that every non-zero element
\[
0 \neq h \in \bigoplus_{\alpha<\mathfrak{c}} H_\alpha
\]
generates a dense subgroup of $G$.
\end{lemma}

\begin{proof}
Let $X$ be the Pontryagin dual of $G$.
Then $X$ is torsion-free and $|X|<\mathfrak{c}$.

Decompose the circle group $\mathbb{T}$ as
\[
\mathbb{T} = t(\mathbb{T}) \oplus \bigoplus_{\alpha<\mathfrak{c}} Q_\alpha,
\]
where $t(\mathbb{T})$ denotes the torsion subgroup of $\mathbb{T}$ and each
$Q_\alpha$ is a $\mathfrak{c}$-dimensional $\mathbb{Q}$-vector space.

For each $\alpha<\mathfrak{c}$, fix a group embedding
\[
b_\alpha : X \longrightarrow Q_\alpha \le \mathbb{T}.
\]
Each $b_\alpha$ can be regarded as an element of $G$.
Since $b_\alpha$ is injective, the cyclic subgroup
\[
H_\alpha := \hull{b_\alpha}
\]
is dense in $G$, and it is straightforward to verify that the family
$\{H_\alpha : \alpha<\mathfrak{c}\}$ is independent.

Now let
\[
0 \neq h = n_1 b_{\alpha_1} + n_2 b_{\alpha_2} + \cdots + n_k b_{\alpha_k}
\in \bigoplus_{\alpha<\mathfrak{c}} H_\alpha.
\]
To show that $\hull{h}$ is dense in $G$, it suffices to verify that $h$,
viewed as a homomorphism $h : X \to \mathbb{T}$, has trivial kernel.

Let $x \in \ker h$. Then
\[
0 = h(x) = \sum_{i=1}^k n_i b_{\alpha_i}(x).
\]
Since $b_{\alpha_i}(x)\in Q_{\alpha_i}$ and the subspaces $Q_{\alpha}$ are
linearly independent, this equality implies $x=0$.
Thus $\ker h=\{0\}$, and the proof is complete.
\end{proof}

\begin{remark}\label{re}
The above lemma immediately implies that for every proper closed subgroup $K$ of $G$,
the family $\{H_\alpha : \alpha<\mathfrak{c}\}\cup\{K\}$ remains independent.
Let $\pi: G\to G/K$ be the canonical projection.
Then the family $\{\pi(H_\alpha) : \alpha<\mathfrak{c}\}$ is still independent.

Therefore, for any subgroup $P\leq G/K$ with $|P|<\mathfrak{c}$,
there exists $\alpha<\mathfrak{c}$ such that
\[
\pi(H_\alpha)\cap P=\{0\}.
\]
Equivalently,
\[
H_\alpha\cap \pi^{-1}(P)\subseteq K.
\]
Since $K\cap H_\alpha=\{0\}$, this implies
\[
H_\alpha\cap \pi^{-1}(P)=\{0\}.
\]
\end{remark}

Now we consider the following question posed in \cite{MT}.

\begin{question}\label{q1}
Is it true that every separable connected compact group is $d$-independent?
\end{question}

Our next goal is to answer this question in the affirmative.
The proof proceeds in two steps.
First, we reduce the problem to the case of Lie groups.
Second, we establish the desired conclusion in this case.

Here, by a \emph{simple connected Lie group} we mean a connected real Lie group without
proper non-trivial connected closed normal subgroups.
Such groups need not be \emph{topologically simple}, as they may have a non-trivial
(but discrete) centre.
Throughout, $Z(G)$ denotes the centre of a group $G$.

\begin{proposition}\label{nonAb}
If every simple connected compact Lie group is $d$-independent,
then every separable connected compact group is $d$-independent.
\end{proposition}

\begin{proof}
Let $H$ be a subgroup of a separable connected compact group $G$ with $|H|<\mathfrak{c}$.
By a well-known structure theorem for connected compact groups \cite[Theorem 9.24]{HM},
\[
G=AG',
\]
where $A$ is the identity component of the centre of $G$,
$G'$ is the derived subgroup of $G$, and
\[
\Delta:=A\cap G'
\]
is totally disconnected.

Moreover, there exists a family $\{S_i : i\in I\}$ of simple connected compact Lie groups
with $|I|\le w(G)\le \mathfrak{c}$, and a quotient homomorphism
\[
\pi:\prod_{i\in I} S_i \longrightarrow G'
\]
such that
\[
\ker\pi \le \pi^{-1}(\Delta) \le \prod_{i\in I} Z(S_i)=:Z.
\]

Let
\[
P=\pi^{-1}(AH\cap G'),
\]
and for each $i\in I$, let $P_i$ be the projection of $P$ onto $S_i$.
We claim that $|P_i|<\mathfrak{c}$ for every $i\in I$.

Indeed,
\[
P/Z \cong (AH\cap G')/\pi(Z)
\]
is a quotient group of
\[
(AH\cap G')/\Delta \cong AH/A,
\]
which has cardinality $<\mathfrak{c}$.
Since each $Z(S_i)$ is finite, it follows that
\[
P\big/\prod_{j\in I\setminus\{i\}} Z(S_j)
\]
also has cardinality $<\mathfrak{c}$.
As the kernel of the projection $P\to P_i$ contains
$\prod_{j\in I\setminus\{i\}} Z(S_j)$, the claim follows.
Hence $|P_i|<\mathfrak{c}$ for all $i$.

By assumption, for each $i\in I$ there exists a countable dense subgroup
$D_i\le S_i$ such that
\[
D_i\cap P_i Z(S_i)=\{1_{S_i}\}.
\]
Let
\[
D=\prod_{i\in I} D_i.
\]
Then $D$ is separable and
\[
D\cap PZ \le D\cap \prod_{i\in I} P_i Z(S_i)=\{1\}.
\]

Since
\[
\pi(D)\cap AH
=\pi\bigl(D\cap \pi^{-1}(AH\cap G')\bigr)
\le \pi(D\cap PZ),
\]
we obtain
\[
\pi(D)\cap AH=\{1_G\}.
\]
Fix a countable dense subgroup $E$ of $\pi(D)$.
If $A$ is trivial, then $H = A H$, and the result follows immediately.  
Henceforth, we assume that $A$ is nontrivial.

Next, consider the quotient homomorphism
\[
\psi:G\to G/G'\cong A/\Delta,
\]
where the isomorphism is given by $xG'=a\Delta$ whenever $x=ag$ with $a\in A$ and
$g\in G'$.
Then
\[
(HG'\cap A)/\Delta \cong HG'/G'
\]
has cardinality $<\mathfrak{c}$.
By Remark~\ref{re}, the group $A$ admits a dense cyclic subgroup $B$ such that
\[
B\cap HG'=\{1_G\}.
\]

Finally, let
\[
L=BE.
\]
Then $L$ is countable and dense in $G$.

Take $x=ag\in H\cap L$ with $a\in B\le A$ and $g\in E\le G'$.
Then
\[
a=xg^{-1}\in HG'\cap B=\{1_G\},
\quad
g=a^{-1}x\in AH\cap E=\{1_G\}.
\]
Hence $H\cap L=\{1_G\}$.
\end{proof}

Note that each $S_i$ is in fact \emph{simply connected}.
Hence, the above proposition reduces Question~\ref{q1} to the following.

\begin{question}
Is every simple, simply connected, compact Lie group $d$-independent?
\end{question}

Let us begin by recalling some elementary facts about Lie groups that will be used throughout this section.
A \emph{maximal torus} of a connected compact Lie group is a maximal subgroup that is topologically isomorphic to a finite power of the circle group $\mathbb{T}$.
It is well known that every element of a connected compact Lie group lies in some maximal torus, and that all maximal tori are conjugate to each other.

We first establish a linear-algebraic lemma, which will later be used to rule out coverings by “too few’’ proper subgroups.

\begin{lemma}\label{vector}
	Let $V$ be a real vector space of dimension $n$, where $1\le n\in\mathbb{N}$.
	Then $V$ cannot be covered by fewer than $\mathfrak{c}$ many subspaces of dimension at most $n-1$.
\end{lemma}

\begin{proof}
	We argue by induction on $n=\dim V$.
	The case $n=1$ is trivial.
	
	Assume that the assertion holds for all vector spaces of dimension at most $n-1$, and let $\dim V=n$.
	Suppose, towards a contradiction, that there exists a family $\{V_i:i\in I\}$ of subspaces of $V$ such that $|I|<\mathfrak{c}$, $\dim V_i\le n-1$ for every $i\in I$, and
	\[
	V=\bigcup_{i\in I}V_i.
	\]
	
	Fix an arbitrary $(n-1)$-dimensional subspace $U$ of $V$, and define $U_i=V_i\cap U$.
	Then
	\[
	U=\bigcup_{i\in I}U_i.
	\]
	By the inductive hypothesis, there exists $i\in I$ such that $U\subseteq U_i$, hence $U\subseteq V_i$.
	
	Since there are $\mathfrak{c}$ many pairwise distinct $(n-1)$-dimensional subspaces of $V$, it follows that some $V_i$ contains at least two distinct $(n-1)$-dimensional subspaces.
	This implies $\dim V_i\ge n$, a contradiction.
\end{proof}

We next record a simple group-theoretic observation concerning intersections of cosets.

\begin{lemma}\label{Inter}
	Let $G$ be a group and let $A,B$ be subgroups of $G$.
	For any $g\in G$, if $A\cap gB\neq\emptyset$, then
	\[
	A\cap gB=a(A\cap B)
	\]
	for some $a\in A$.
\end{lemma}

\begin{proof}
	Choose $a\in A\cap gB$.
	Then $aB=gB$, and therefore
	\[
	A\cap gB=A\cap aB=aA\cap aB=a(A\cap B).\qedhere
	\]
\end{proof}

Combining the previous two lemmas, we obtain the following covering result for connected Lie groups.
\begin{corollary}\label{Coro:coset}
	Let $G$ be a connected Lie group.
	If $\{H_i:i\in I\}$ is a (not necessarily injectively indexed) family of proper closed subgroups and if there exist elements $g_i\in G$ such that
	\[
	G=\bigcup_{i\in I}g_iH_i,
	\]
	then $|I|\ge\mathfrak{c}$.
\end{corollary}

\begin{proof}
	Assume $|I|<\mathfrak{c}$.
	Let $\mathfrak{g}$ be the Lie algebra of $G$, and denote by $\exp_G:\mathfrak{g}\to G$ the exponential map.
	
	For each $i\in I$, set
	\[
	\mathfrak{h}_i=\{X\in\mathfrak{g}:\exp_G(\mathbb{R}X)\subseteq H_i\}.
	\]
	Then $\mathfrak{h}_i$ is the Lie algebra of $H_i$ (cf. \cite[Section 8.3]{HN}).
	Since $H_i$ is proper, $\mathfrak{h}_i\ne\mathfrak{g}$.
	
	Fix a non-zero element $X\in\mathfrak{g}$.
	Since
	\[
	\exp_G(\mathbb{R}X)\subseteq \bigcup_{i\in I}g_iH_i
	\]
	and $|I|<\mathfrak{c}$, there exists $i\in I$ such that
	$\exp_G(\mathbb{R}X)\cap g_iH_i$ is uncountable.
	
	By Lemma~\ref{Inter},
	\[
	\exp_G(\mathbb{R}X)\cap g_iH_i=\exp_G((t+D)X)
	\]
	for some $t\in\mathbb{R}$ and some uncountable subgroup $D\le\mathbb{R}$.
	Since $t+D$ is dense in $\mathbb{R}$ and $g_iH_i$ is closed, we conclude that
	\[
	\exp_G(\mathbb{R}X)\subseteq g_iH_i,
	\]
	hence $g_iH_i=H_i$ and $X\in\mathfrak{h}_i$.
	
	Thus,
	\[
	\mathfrak{g}=\bigcup_{i\in I}\mathfrak{h}_i,
	\]
	contradicting Lemma~\ref{vector}.
\end{proof}

We now apply the preceding results to conjugacy classes of elements in compact Lie groups.

\begin{lemma}\label{Le:coset}
	Let $G$ be a connected compact Lie group.
	If $x\in G$ is non-central and $T$ is a maximal torus of $G$, then
	\[
	X:=\{g\in G:\; gxg^{-1}\in T\}
	\]
	is a finite union of cosets of some proper closed subgroup of $G$.
\end{lemma}

\begin{proof}
	We first treat the case where $x\in T$.
	
	Let
	\[
	N_G(T)=\{g\in G:\; gT=Tg\}
	\]
	be the normalizer of $T$ in $G$, and let
	\[
	W(G,T):=N_G(T)/T
	\]
	denote the (analytic) Weyl group of $G$ relative to $T$.
	It is well known that $W(G,T)$ is finite
	(cf.\ \cite[Theorem~6.36(e)]{Sep}).
	Hence there exists a finite subset $F\subseteq N_G(T)$ such that
	\[
	N_G(T)=FT.
	\]
	
	Let $g\in X$.
	Since both $gxg^{-1}$ and $x$ belong to $T$, it follows from
	\cite[Theorem~6.36(f)]{Sep} that there exists $h\in F$ such that
	\[
	hxh^{-1}=gxg^{-1}.
	\]
	Equivalently,
	\[
	h^{-1}gx=xh^{-1}g,
	\]
	which means that
	\[
	h^{-1}g\in Z_G(x):=\{y\in G:\; xy=yx\}.
	\]
	
	Since $x$ is non-central, its centralizer $Z_G(x)$ is a proper closed
	subgroup of $G$.
	Consequently,
	\[
	X=\bigcup_{h\in F} hZ_G(x)
	\]
    is the union of finite cosets of the proper closed subgroup $Z_G(a)$.	
	
	We now consider the general case.
	Choose $h\in G$ such that $x\in hTh^{-1}=:S$, which is again a maximal torus of $G$.
	Then
	\[
	gxg^{-1}\in T \quad\text{if and only if}\quad (hg)x(hg)^{-1}\in S.
	\]
	Applying the above argument to the maximal torus $S$, we see that the set $hX$
	is a finite union of cosets of the proper closed subgroup $Z_G(x)$.
	It follows that $X$ itself has the same property.
\end{proof}

We are now ready to establish the main result of this section, showing that
$d$-independence holds for all connected compact non-abelian Lie groups.

\begin{proposition}\label{noch}
	Every connected compact non-abelian Lie group is $d$-independent.
\end{proposition}

\begin{proof}
	Let $G$ be a connected compact non-abelian Lie group, and let $T$ be a maximal torus of $G$.
	By \cite[Corollary~6.87]{HM}, the group $G$ admits a dense free subgroup $F$ generated by two elements.
	Clearly, for every $g\in G$, the conjugate subgroup $gFg^{-1}$ is again a dense free subgroup of $G$.
	
	The key point is to show that one can choose such a conjugate so that it avoids
	any prescribed “small’’ subset of $G$.
	
	\begin{claim}\label{cl1}
		For every subset $C\subseteq G$ with $|C|<\mathfrak{c}$, the set
		\[
		\Bigl\{g\in G:\; gFg^{-1}\cap \bigcup_{c\in C} cTc^{-1}\neq \{1\}\Bigr\}
		\]
		is a proper subset of $G$.
	\end{claim}
	
	\begin{proof}[Proof of Claim~\ref{cl1}]
		We may write
		\[
		\Bigl\{g\in G:\; gFg^{-1}\cap \bigcup_{c\in C} cTc^{-1}\neq \{1\}\Bigr\}
		=
		\bigcup_{a\in F\setminus\{1\}}\;
		\bigcup_{c\in C}
		\{g\in G:\; gag^{-1}\in cTc^{-1}\}.
		\]
		
		Fix $a\in F\setminus\{1\}$ and $c\in C$.
		By Lemma~\ref{Le:coset}, the set
		\[
		\{g\in G:\; gag^{-1}\in cTc^{-1}\}
		\]
		is a finite union of cosets of some proper closed subgroup of $G$.
		Consequently, the entire set in question is a union of fewer than $\mathfrak{c}$
		many cosets of (possibly different) proper closed subgroups of $G$.
		
		Note that $F$ is centre-free, hence no $a\in F\setminus\{1\}$ can be central in $G$. By Corollary~\ref{Coro:coset}, such a union cannot cover $G$.
		Hence the above set is a proper subset of $G$, as claimed.
	\end{proof}
	
	Now let $H$ be an arbitrary subgroup of $G$ with $|H|<\mathfrak{c}$.
	Since every element of $G$ is contained in some maximal torus and all maximal tori
	are conjugate, there exists a subset $C\subseteq G$ with $|C|<\mathfrak{c}$ such that
	\[
	H\subseteq \bigcup_{c\in C} cTc^{-1}.
	\]
	
	By Claim~\ref{cl1}, the set
	\[
	\Bigl\{g\in G:\; gFg^{-1}\cap \bigcup_{c\in C} cTc^{-1}\neq \{1\}\Bigr\}
	\]
	is a proper subset of $G$.
	Therefore, we may choose $g\in G$ such that
	\[
	gFg^{-1}\cap \bigcup_{c\in C} cTc^{-1}=\{1\}.
	\]
	
	Since $H$ is contained in $\bigcup_{c\in C} cTc^{-1}$, it follows that
	\[
	gFg^{-1}\cap H=\{1\}.
	\]
	Finally, note that $gFg^{-1}$ is countable and dense in $G$.
	This shows that $G$ is $d$-independent.
\end{proof}

Combining Propositions \ref{nonAb} and \ref{noch}, we obtain the main result:
\begin{theorem}
	Every separable connected compact group is $d$-independent.
\end{theorem}

\section*{Acknowledgements}
The authors express sincere gratitude to Professor Wei He for insightful discussions. This work was supported by NSFC Grants 12301089 and 12271258.

\end{document}